\documentclass[11pt]{amsart}
\usepackage{amssymb}
\usepackage{dsfont}

\newtheorem{theorem}{Theorem}[section]
\newtheorem{lemma}[theorem]{Lemma}
\newtheorem{proposition}[theorem]{Proposition}
\newtheorem{problem}[theorem]{Problem}
\newtheorem{corollary}[theorem]{Corollary}

\newtheorem{definition}[theorem]{Definition}
\newtheorem{remark}[theorem]{Remark}

\newcommand{\B}{\mathcal B}
\newcommand{\N}{\mathbb N}
\newcommand{\Q}{\mathbb Q}
\newcommand{\R}{\mathbb R}
\newcommand{\C}{\mathbb C}

\newcommand{\Co}{\mathfrak c}

\newcommand{\on}{\operatorname}
\author{Artur Bartoszewicz}
\address{Institute of Mathematics, \L \'od\'z University of Technology,
W\'olcza\'nska 215, 93-005 \L \'od\'z, Poland}
\email {arturbar@p.lodz.pl}
\author{Marek Bienias}
\address{Institute of Mathematics, \L \'od\'z University of Technology,
W\'olcza\'nska 215, 93-005 \L \'od\'z, Poland}
\email {marek.bienias88@gmail.com}
\author{Ma\l gorzata Filipczak}
\address{Faculty of Mathematics and Computer Science, University of \L \'od\'z, Banacha 22, 93-008 \L \'od\'z, Poland}
\email {malfil@math.uni.lodz.pl}
\author{Szymon G\l \c ab}
\address{Institute of Mathematics, \L \'od\'z University of Technology,
W\'olcza\'nska 215, 93-005 \L \'od\'z, Poland}
\email {szymon.glab@p.lodz.pl}
\title[Strong $\Co$-algebrability of strong Sierpi\'nski-Zygmund functions]{Strong $\Co$-algebrability of strong Sierpi\'nski-Zygmund, smooth nowhere analytic and other sets of functions}
\date{}

\thanks{The first, the second and the last authors have been supported by the Polish Ministry of Science and Higher Education Grant No IP2011 014671 (2012--2014).}
\subjclass[2010]{Primary: 15A03; Secondary: 46J10, 26A15} 
\keywords{Algebrability, strong algebrability, Sierpi\'nski-Zygmund function, nowhere H\"older function, approximately continuous function, nowhere monotone function, Bruckner-Garg function, Baire hierarchy}
\date{}

\begin{document}

\begin{abstract}
We present a useful technique of proving strong $\Co$-algebrability. As an outcome we obtain the strong $\Co$-algebrability of the following sets of functions: strong Sierpi\'nski-Zygmund, nowhere H\"older, Bruckner-Garg, nowhere monotone differentiable, a certain Baire class, smooth and nowhere analytic functions.
\end{abstract}

\maketitle

\section{Introduction}

In the last $10$ years there appeared a new point of looking at the largeness of sets included in some function spaces. Let us recall the following notion, that  appeared for the first time in \cite{2} and has its origins in the works of  R.M. Aron, V.I. Gurariy, D. P\'erez-Garc\'ia, J.B. Seoane-Sep\'ulveda (see \cite{2,3,4}). Recently, there were published two surveys in this topic, that contain several examples (see \cite{s1,s2}).

\begin{definition}
Let $\kappa$ be a cardinal number.
\begin{enumerate}
\item Let $\mathcal{L}$ be a vector space and a set $A\subseteq\mathcal{L}.$ We say that $A$ is $\kappa$-lineable if $A\cup \{0\}$ contains a $\kappa$-dimensional vector space;
\item Let $\mathcal{L}$ be a Banach space and a set $A\subseteq\mathcal{L}.$ We say that $A$ is spaceable if $A\cup \{0\}$ contains an infinite dimensional closed vector space;
\item Let $\mathcal{L}$ be a linear commutative algebra and a set $A\subseteq\mathcal{L}.$ We say that $A$ is $\kappa$-algebrable if $A\cup \{0\}$ contains a $\kappa$-generated algebra $B$ (i.e. the minimal system of generators of $B$ has cardinality $\kappa$).
\end{enumerate}
\end{definition}

A. Bartoszewicz and S. G\l \c{a}b in \cite{6} went further and asked for existence of free structures inside some set $A\cup\{0\}.$ They introduced the notion of strong algebrability.

\begin{definition}{\cite{6}}
Let $\kappa$ be a cardinal number. Let $\mathcal{L}$ be a linear commutative algebra and a set $A\subseteq\mathcal{L}.$ We say that $A$ is strongly $\kappa$-algebrable if $A\cup \{0\}$ contains a $\kappa$-generated algebra $B$ that is isomorphic with a free algebra (denote by $X=\{x_\alpha: \alpha<\kappa\}$ the set of generators of this free algebra).
\end{definition}

Remark, that the set $X=\{x_\alpha: \alpha<\kappa\}$ is the generating set of some free algebra contained in $A\cup\{0\}$ if and only if the set $\tilde{X}$ of elements of the form $x_{\alpha_1}^{k_1}x_{\alpha_2}^{k_2}\cdots x_{\alpha_n}^{k_n}$ is linearly independent and all linear combinations of elements from $\tilde{X}$ are in $A\cup\{0\}$; equivalently for any $k\in \N,$ any nonzero polynomial $P$ in $k$ variables without a constant term and any distinct $y_1,...,y_k\in X,$ we have $P(y_1,...,y_k)\neq 0$ and $P(y_1,...,y_k)\in A.$ This is a very natural way of proving algebrability, and therefore many authors, while proving algebrability, get strong algebrability actually. 

It is easy to check that for any cardinal number $\kappa$, the following implications hold.
$$\kappa \text{-strong algebrability} \Rightarrow \kappa \text{-algebrability} \Rightarrow \kappa \text{-lineability.}$$
Moreover, since every infinite dimensional Banach space has a linear base of cardinality $\Co$, then
$$\text{spaceability} \Rightarrow \Co \text{-lineability}.$$

On the other hand, none of these implications can be reversed (cf.  \cite{11, BBGS}).

Research on lineability, spaceability and algebrability can be viewed as a research on smallness of sets. The notions of smallness, like a measure zero, meager and $\sigma$-porosity, are connected with measure, topology and metric structure, respectively. The set that is not lineable, spaceable or algebrable can be viewed as small, since it does not contain a large algebraic structure. It means that being nonlineable, nonspaceable or nonalgebrable is an algebraic notion of smallness. The above named notions of smallness are not only different from the algebraic ones, but they are also incomparable. The following example shows, that largeness in algebrability does not coincide with largeness in topology. In \cite{GQ} it was proved that the set of all functions from $C[0,1]$ which attain their maximum at exactly one point and attain their minimum at exactly one point is not $2$-lineable. On the other hand, this set is residual. In \cite{6} it was shown that there is a subset of $\ell_1$ which is residual, strongly $\Co$-algebrable, but not spaceable. It is well known that every proper closed infinitely dimensional subspace of a Banach space is porous, which shows that a spaceable set can be very small from topological and metric point of view.

In this paper we prove several results in strong $\Co$-algebrability. Almost all the presented below results are the best possible in the sense of complexity of the structure and cardinality of the set of generators. Let us remark, that the he $\Co$-algebrability of functions families $\mathcal F$, with $\vert\mathcal F\vert=\Co$, that is maximal algebrability, was also proved in a few papers -- \cite{11}, \cite{GKP} and \cite{GGMS}.  
%we recall an idea of M. Balcerzak, A. Bartoszewicz and M. Filipczak from \cite{BBF}, and using their criterion, 

The paper is organized as follows. In Section 2 we prove that family of all strongly Sierpi\'nski-Zygmund functions is strongly $\Co$-algebrable  provided this family is nonempty. Strongly Sierpi\'nski-Zygmund functions are classical Sierpi\'nski-Zygmund functions under the Continuum Hypothesis, they do not exist under Martin's Axiom ($\on{MA+\neg CH}$), and there are strongly Sierpi\'nski-Zygmund functions in some models of ZFC in which the Continuum Hypothesis fails. In Section 3 we prove strong $\Co$-algebrability of family of all functions $f:\R\to\R$ which set of continuity points equals to fixed $G_\delta$ set. Our result completes the knowledge on algebrability of such functions families. In Section 4 we prove strong $\Co$-algebrability of the set of approximately continuous and discontinuous almost everywhere, nowhere H\"older functions, nowhere analytic smooth functions, differentiable nowhere monotone functions and Bruckner-Garg functions. All these results improve known facts, proved by other authors. In Section 5 we present strong $\Co$-algebrability of Darboux Baire class one functions which are not approximately continuous, Baire class one functions which are not Darboux, and functions which are in Baire class $\beta$ but are nowhere Baire class $\alpha$ for $\alpha<\beta$. We prove also that the set of all sequences which set of limit points is homeomorphic to the ternary Cantor set is strongly $\Co$-algebrable. In the last section we give some final remarks and state some open questions.

It is a simple observation that the set $\{x\mapsto\exp(rx):r\in\R\}$ is linearly independent family of functions in $\R^\R$. Moreover, if the set $\{r_\alpha:\alpha<\Co\}$ is linearly independent over $\Q$, then $\{x\mapsto\exp(r_\alpha x):\alpha<\Co\}$ is the set of free generators. Such families of exponential functions were found useful by several authors in solving lineability problems. They used the following strategy. Take one function $F$ from considered family $\mathcal{A}$ and prove that each linear combination $\sum_{i=1}^ka_i\exp(r_ix)F(x)$ is still in $\mathcal{A}$. Then $\mathcal{A}$ is $\Co$-lineable. L. Bernal-Gonz\'alez and M. Cabrera in \cite{BG2}, using this idea, proved that subset of functions in $\on{CBV}[0,1],$ which are absolutely continuous on no subinterval of $[0, 1],$ is maximal dense-lineable. P. Jim\'enez-Rodr\'{\i}guez, G. Mu\~noz-Fern\'andez and J. Seoane-Sep\'ulveda in \cite{JMS} proved that the set of differentiable functions $f:\R\to\R^2$ that do not enjoy the Mean Value Theorem is $\Co$-lineable.  J. G\'amez-Merino, G. Mu\~noz-Fern\'andez, V. S\'anchez, J. Seoane-Sep\'ulveda in \cite{8} proved that the set of bounded approximately continuous mappings defined on $\R$, that are discontinuous almost everywhere, is $\Co$-lineable. 

In \cite{GGMS} the authors used composition instead of multiplication to prove $\mathfrak c$-algebrability of the set of continuous functions such that both of their sets of proper local minima and maxima are dense in $\R$. Recently this idea has been continued in \cite{BBF}, where was proposed the following setting.

\begin{definition}\cite{BBF}
We say that a function $f:\R\to\R$ is \emph{exponential-like} (of rank $m$) whenever $f$ is given by $f(x)=\sum_{i=1}^m a_i e^{\beta _i x}$ 
for some distinct nonzero real numbers $\beta_1,...,\beta_m$ and some nonzero real numbers $a_1,...,a_m$.
We will also consider exponential-like functions (of the same form) with the domain $[0,1]$.
\end{definition}

In \cite{BBF} the authors proved a simple, but very useful property of exponential-like functions.

\begin{lemma}\label{expLem}
For every positive integer $m$, any exponential-like function $f: \R\to \R$  of rank $m$, and
each $c\in\R$, the preimage $f^{-1}[\{c\}]$ has at most $m$ elements. Consequently, $f$ is not constant in
every subinterval of $\R$. In particular there exists a finite decomposition of $\R$ into intervals, such that $f$ is strictly monotone on each of them.
\end{lemma}

The criterion described in \cite{BBF} is the following.

\begin{theorem}\label{method BBF}
Let $\mathcal{F}\subseteq {\R}^{[0,1]}$ and assume that there exists a function $F\in\mathcal{F}$ such that
$f\circ F\in\mathcal{F}\setminus\{ 0\}$ for every exponential-like
function $f:\R\to\R$. Then $\mathcal{F}$ is strongly $\Co$-algebrable. More exactly,
if $H\subseteq \R$ is a set of cardinality $\Co$ and linearly independent over the rationals $\Q$,
then $\exp\circ\,(rF)$, $r\in H$, are free generators of an algebra contained in $\mathcal{F}\cup\{ 0\}$.
\end{theorem}

For the reader's convenience we recall the simple proof of this fact.

\begin{proof}
Consider a set $H$ of cardinality $\Co$, such that it is a linearly independent over $\Q.$
By the assumption we have that $\{\exp\circ\,(rF):r\in H\}\subseteq\mathcal{F}.$ 
To show that it is a set of free generators, consider $n\in\N$ and a non-zero polynomial $P$ in $n$ variables without a constant term.
The function given by $[0,1] \ni x\mapsto P(e^{r_1 F(x)},e^{r_2 F(x)},...,e^{r_n F(x)})$
is of the form

$$\sum_{i=1}^m a_i \left(e^{r_1 F(x)}\right)^{k_{i1}} \left(e^{r_2 F(x)}\right)^{k_{i2}}...\left(e^{r_n F(x)}\right)^{k_{in}}
=\sum_{i=1}^m a_i \exp\left(F(x)\sum_{j=1}^n r_j k_{ij}\right)$$
where $a_1,..., a_m$ are nonzero real numbers and the matrix $[k_{ij}]_{i\le m,j\le n}$ of nonnegative integers has distinct nonzero rows.

Since the function $t\mapsto\sum_{i=1}^m a_i \exp(t\sum_{j=1}^n r_j k_{ij})$ is exponential-like, the function $[0,1] \ni x\mapsto P(e^{r_1 F(x)},e^{r_2 F(x)},...,e^{r_n F(x)})$
is in $\mathcal{F}\setminus\{ 0\}$.
\end{proof}
In all proofs we make use of Theorem \ref{method BBF}.

\section{Strongly Sierpi\'nski-Zygmund functions}

In this section we will consider a type of functions that firstly appeared in 1920's (see \cite{SZ}). The classical Luzin's Theorem implies that for every measurable function $f:\R\to\R$ there is a set $S\subseteq \R$ with infinite measure such that $f|_S$ is continuous. In 1922 (see \cite{Blu}) H. Blumberg showed that if we omit the assumption that $f:\R\to \R$ is measurable, then some weaker version of the assertion of the Luzin's Theorem remains true. More precisely H. Blumberg proved the following.  
\begin{theorem}\cite{Blu}
Let $f:\R\to\R$ be an arbitrary function, then there exists a dense subset $S\subseteq\R$ such that $f|_S$ is continuous.
\end{theorem}
In the proof the set $S$ was countable, so one can ask if the set $S$ can be uncountable. The partial answer to this was given by W. Sierpi\'nski and A. Zygmund in \cite{SZ}.
\begin{theorem}\label{SZTh}
There exists a function $f:\R\to\R$, such that for any set $Z\subseteq\R$ of cardinality $\Co$, the restriction $f|_Z$ is not a Borel map.
\end{theorem}
Obviously, under the Continuum Hypothesis the restriction of this function to any uncountable set cannot be continuous. J. Shinoda (see \cite{Sh}) proved that under Martin's Axiom and negation of the Continuum Hypothesis for every function $f:\R\to\R$ there exists an uncountable set $Z\subseteq \R$, such that $f|_Z$ is continuous. By classical theorems of Luzin and Nikodym, a function from Theorem \ref{SZTh} is nonmeasurable and does not have Baire property. Let us recall the following notion.
\begin{definition}
We say that a function $f:\R\to\R$ is a Sierpi\'nski-Zygmund function, if for every set $A\subseteq \R$ of cardinality $\Co,$ the restriction $f|_A$ is not a Borel map.
\end{definition}

The set of all Sierpi\'nski-Zygmund functions $f:\R\to\R$ (let us denote it by $\mathcal{SZ}(\R)$) was firstly considered, in the context of algebrability, by  J.L. G\'amez-Merino, G.A. Munoz-Fern\'andez, V.M. S\'anchez and J.B. Seoane-Sep\'ulveda in \cite{8}. The authors proved that this set is $\Co^+$-lineable and, also, $\Co$-algebrable. 
A. Bartoszewicz, S. G\l\c{a}b, D. Pellegrino and J.B. Seoane-Sep\'ulveda in \cite{7} showed the following.

\begin{theorem}
The set $\mathcal{SZ}(\R)$ is strongly $\kappa$-algebrable, provided there exists an almost disjoint family in $\Co$ of cardinality $\kappa,$ (i.e. there is a family $\{A_\alpha: \alpha<\kappa\}\subseteq [\Co]^\Co$ with $|A_\alpha \cap A_\beta|<\Co$ for any $\alpha\neq \beta$).
\end{theorem}

It is known that in ZFC there is an almost disjoint family in $\Co$ of cardinality greater than $\Co.$ So it was the first time, when the strong algebrability was at the level greater than $\Co.$ On the other hand, J. G\'{a}mez-Merino, J. Seoane-Sep\'{u}lveda in \cite{GamSep} showed that there is a model of ZFC in which there is no almost disjoint family in $\Co$ of cardinality $2^\Co.$

Observe here the following property.
\begin{remark}\cite{GamSep}
Any additive group $\mathcal{A}\subseteq\mathcal{SZ}(\R)$ generates an almost disjoint family in the plane $\R\times\R$ (by considering graphs of $f\in\mathcal{A}$ as a members of this family).
\end{remark}

Hence, the result of A. Bartoszewicz, S. G\l\c{a}b, D. Pellegrino and J.B. Seoane-Sep\'ulveda is the best possible in ZFC.

Let us introduce the following notion.

\begin{definition}
We say that a function $f:\R\to\R$ is a strong Sierpi\'nski-Zygmund function, if for every set $A\subseteq \R$ of cardinality $\omega_1$ the restriction $f|_A$ is not a Borel map. Let us denote by $s\mathcal{SZ}(\R)$ the set of all strong Sierpi\'nski-Zygmund functions.
\end{definition}

Under the negation of the Continuum Hypothesis, being a strong Sierpi\'nski-Zygmund is stronger requirement than being Sierpi\'nski-Zygmund function.

Due to Theorem \ref{method BBF} we may obtain the following.

\begin{theorem}\label{sSZ}
If $s\mathcal{SZ}(\R)\neq \emptyset$, then it is strongly $\Co$-algebrable.
\end{theorem}

\begin{proof}
Let $F\in s\mathcal{SZ}(\R)$ that means if $A\subseteq \R$, $g:\R\to\R$ is Borel and $f\vert_A=g\vert_A$, then $\vert A\vert<\omega_1$.
Let $f:\R\to\R$ be an exponential-like function. We will prove that $f\circ F$ is a strong Sierpi\'nski-Zygmund. Suppose, on the contrary, it is not strong Sierpi\'nski-Zygmund. Then there exists a set $B\subseteq\R$ of cardinality $\omega_1$ and a Borel mapping $g:\R\to\R$ such that $f\circ F|_B=g|_B.$
Notice that $|F(B)|=\omega_1.$ Indeed, suppose that $|F(B)|\leq\omega$, then by the Pigeonhole Principle there is $y \in \R$ such that $|F^{-1}(\{y\})|=\omega_1$. But then $F|_{F^{-1}(\{y\})}$ is constant, hence Borel on the set $F^{-1}(\{y\})$ of cardinality $\omega_1$ that yields us to a contradiction (since $F$ is strong Sierpi\'nski-Zygmund).\\
Moreover, by Lemma \ref{expLem}, there exists (again by the Pigeonhole Principle) an interval $(\alpha,\beta)$ such that
\begin{itemize}
\item[(1)] $f|_{(\alpha,\beta)}$ is strictly monotone, and
\item[(2)] $|(\alpha,\beta)\cap F(B)|=\omega_1.$
\end{itemize}
Let $B_1=\{t\in B: F(t)\in (\alpha,\beta)\}.$ By (2) we have $|B_1|=\omega_1$. By (1) the function $f$ is invertible on $(\alpha,\beta)$ and therefore
$$F(t)=f^{-1}\circ f \circ F(t) = f^{-1} \circ g(t)\text{, for }t\in B_1.$$ 
We reach a contradiction, since $ f^{-1} \circ g$ is Borel. 
\end{proof}

One can ask, if the assumption of the above Theorem can be fulfilled under the negation of the Continuum Hypothesis.  Gruenhage proved that if model $V$ of ZFC satisfies $2^\omega=\kappa$ and $\lambda\geq\kappa$ is a cardinal with $\lambda^\omega=\lambda$, then, after adding $\lambda$ Cohen or random reals to $V$, there exists $s\mathcal{SZ}$ function in the extension (for details see \cite{Rec}). 

\section{Functions that are continuous on a fixed $G_\delta$ set}

It is well known that any real function is continuous on a $G_\delta$ set. Several authors considered lineability and coneability of families of functions with prescribed sets of discontinuity points. F. Garc\'ia-Pacheco, N. Palmberg and J. Seoane-Sep\'ulveda in \cite{11} showed $\omega$-lineability of the set of functions with finitely many points of continuity. 
A. Aizpuru,  C. P\'erez-Eslava, F. Garc\'ia-Pacheco and J. Seoane-Sep\'ulveda established in \cite{APGS} that the set of all functions $f:\R \rightarrow \R$ which are continuous only at the points of a fixed open set $U$ (a fixed $G_\delta$ set, respectively) is lineable (coneable). Moreover, A. Bartoszewicz, M. Bienias and S. G\l\c{a}b proved $2^\Co$-algebrability of the family of functions whose sets of continuity points equal to $K$ for a fixed closed set $K\subsetneq\R$ (or $K\subsetneq\C$) \cite{Bern}. Let $G\subseteq [0,1]$ be a $G_\delta$ set and consider the set $\mathcal{C}_G$ of all functions $f:[0,1]\to\R$ which set of continuity points is exactly $G.$

Let us introduce the following notion.
\begin{definition}
We say that a set $A\subseteq[0,1]$ is $\Co$-dense in itself, if for any open interval $I$ we have either $I\cap A=\emptyset$ or $|I \cap A|=\Co.$
\end{definition}

A. Bartoszewicz, S. G\l \c{a}b and A. Paszkiewicz obtained the following characterization.
\begin{theorem}\cite{BGP}
The following conditions are equivalent.
\begin{itemize}
\item[$(i)$] $\mathcal{C}_G$ is strongly $2^\Co$-algebrable;
\item[$(ii)$] $\mathcal{C}_G$ is $\Co^+$-lineable;
\item[$(iii)$] $\R\setminus G$ is $\Co$-dense in itself.
\end{itemize}
\end{theorem}

Moreover, in \cite{BGP}, the authors obtained strong $\Co$-algebrability of $\mathcal{C}_G,$ for some sets $G$ and asked a question about strong $\Co$-algebrability of $\mathcal{C}_{\R\setminus \mathbb{Q}}.$
In this paper we fully answer for the question about strong $\Co$-algebrability of $\mathcal{C}_G$ in the general case (i.e. for any $G_\delta$ set $G\subseteq [0,1]$).
The next result is, very probably, a folklore. 
\begin{proposition}\label{Gdelta}
There exists a function $F\in\mathcal{C}_G$ which has infinitely many limits at each point of its discontinuity.
\end{proposition}

The construction of the function from the assertion of Proposition \ref{Gdelta} was presented for example in \cite{BGP}.
By Theorem \ref{method BBF} we can obtain the following result.

\begin{theorem}
The set $\mathcal{C}_G$ is strongly $\Co$-algebrable.
\end{theorem}

\begin{proof}
Let $F\in \mathcal{C}_G$ be a function like in Proposition \ref{Gdelta}. Let $f:[0,1]\to\R$ be an exponential-like function of rank $m$. We will show that $f\circ F\in\mathcal{C}_G\setminus\{ 0\}.$ Take a point $x\in [0,1]$. We have the following
\begin{itemize}
\item[$1.$] If $F$ is continuous at $x$, then clearly $f\circ F$ is also continuous at $x$;
\item[$2.$] If $F$ is not continuous at $x$, then there are sequences (cf. Proposition \ref{Gdelta}) $(t^{(k)}_n)_{n\in\N}$ for $k\in\{1,...,m+1\}$ such that for all $k$ we have $t^{(k)}_n \to x$ and $F(t^{(k)}_n)\to y^{(k)}$ with $y^{(k)}\neq y^{(l)}$ for $k\neq l.$ Then $f\circ F(t^{(k)}_n)\to f(y^{(k)}).$ By the Lemma \ref{expLem} we have $f(y^{(k)})\neq f(y^{(l)})$ for some $k\neq l$ so $f\circ F$ is not continuous at $x.$
\end{itemize}
Hence, for every exponential-like function $f,$ we have $f\circ F\in\mathcal{C}_G\setminus\{ 0\}$ so, by Theorem \ref{method BBF}, we get a strong $\Co$-algebrability of $\mathcal{C}_G.$
\end{proof}

\section{Continuous functions with other interesting properties}

This section will be devoted to different classes of real functions connected with continuity. Let us recall the following notion.
\begin{definition}
We say that a function $f:\R\to\R$ is
\begin{itemize}
\item of the class $\mathcal{C}_0$, provided it is continuous with respect to the natural topology both in the domain and in the range;
\item of the class $\mathcal{C}_n$, provided it is differentiable $n$ times, for $n\in\N;$
\item approximately continuous (denote by $f\in \mathcal{AP}$) provided it is continuous with respect to the classical density topology in the domain and natural topology in the range;
\item of the class $\mathcal{C}_{\infty}$, provided it is smooth (i.e. it has derivatives of all orders);
\item analytic (denote by $f\in\mathcal{A}$), provided it is locally given by a convergent power series.
\end{itemize}
\end{definition}

It is well known that the following inclusions
$$\mathcal{AP} \supseteq \mathcal{C}_0 \supseteq \mathcal{C}_1 \supseteq ... \supseteq \mathcal{C}_{n-1} \supseteq \mathcal{C}_n \supseteq ... \supseteq \mathcal{C}_\infty \supseteq \mathcal{A}$$
hold and are strict. Moreover all of the differences have cardinality $\Co,$ so it is a natural question how big in the sense of algebrability they are.
In this section, using exponential-like functions, we show that all of them are strongly $\Co$-algebrable.

\subsection{Inclusion $\mathcal{AP} \supsetneq \mathcal{C}_0$}

Approximate continuity is a generalization of the concept of continuity, in which the ordinary limit is replaced by an approximate limit. A real function $f$ is approximately continuous at a point $x$, if and only if the approximate limit of $f$ at $x$ exists and equals $f(x),$  which means, there exists a Lebesgue measurable set $E$ such that $$\lim_{h\to 0^+} \frac{\lambda(E\cap [x-h,x+h]}{2h}=1$$ and $f|{E\cup \{x\}}$ is a continuous function. A. Denjoy (see \cite{Den}) proved in 1915 that any Lebesgue measurable function is approximately continuous almost everywhere. In 1952, O. Haupt and Ch. Pauc (see \cite{HP}) introduced a notion of density topology $\tau_d.$ For a nonempty measurable set $E,$ $E\in \tau_d$, if and only if $\lim_{h\to 0^+} \frac{\lambda(E\cap [x-h,x+h]}{2h}=1,$ for any $x\in E.$ A function $f$ is approximately continuous, if and only if it is continuous with respect to the density topology in the domain and natural topology in the range.

In \cite{8}, the authors proved, that the set of bounded approximately continuous mappings defined on $\R$ and discontinuous almost everywhere is $\Co$-lineable. Here, we will take this result to the highest possible level of algebrability.

\begin{proposition}{\cite{Bru}}\label{aprox}
Let $E\subseteq [0,1]$ be a dense $G_\delta$ set of measure zero. There exists a bounded approximately continuous function $F:[0,1]\to[0,\infty)$ such that $F^{-1}(\{0\})=E.$ In particular $F$ has the Darboux property.
\end{proposition}

Observe here, that the function $F$ from Proposition \ref{aprox}, is discontinuous outside the set $E$ (i.e. is discontinuous almost everywhere). Now we are able to state the following.

\begin{theorem}\label{ap}
The set of approximately continuous functions that are discontinuous almost everywhere is strongly $\Co$-algebrable.
\end{theorem}

\begin{proof}
Let $E\subseteq [0,1]$ be a dense $G_\delta$ set of measure zero and $F:[0,1]\to[0,\infty)$ be like in Proposition \ref{aprox}. Let $f:\R\to\R$ be an exponential-like function. Clearly the function $f\circ F$ is approximately continuous. We will show that it is discontinuous on the set $[0,1]\setminus E.$
If $(f\circ F)(x)\neq p:= f(0),$ for some $x\in [0,1],$ then $f\circ F$ is not continuous at $x.$ Indeed, there exists a sequence $(x_k)_{k\in\N}$ of elements of $E$ convergent to $x$ and then, by the definition of $F$, we have that $(f\circ F)(x_k)=p$ for any $k\in \N.$

Hence, it is enough to show that $f\circ F$ is discontinuous on the set $(f\circ F)^{-1}(\{p\})\setminus E.$ To see this, fix $x\in (f\circ F)^{-1}(\{p\})\setminus E,$ then $F(x)=s_0>0.$ By Lemma \ref{expLem}, there exists $s\in(0,s_0)$ with $f(s)\neq p.$ Moreover, for any $k\in\N,$ there is $z\in E \cap (x,x+\frac{1}{k}).$ By the Darboux property of $F,$ one can find $x_k \in (x,z)$ with $F(x_k)=s.$ Therefore, there is a sequence $(x_k)_{k\in\N}$ convergent to $x,$ such that $F(x_k)=s$ for every $k\in \N.$ Hence, $(f\circ F)(x_k)$ tends to $f(s)\neq p,$ so $(f\circ F)$ is not continuous at $x.$
This means, that $f\circ F$ is approximately continuous and discontinuous almost everywhere. Hence, by Theorem \ref{method BBF} we are done.
\end{proof}

Let us point out, that the function $F$ in the proof is bounded. Moreover, for any exponential-like function $f,$ the composition $f\circ F$ is bounded too. Hence, we may remark the following.

\begin{remark}
The set of bounded approximately continuous functions that are discontinuous almost everywhere is strongly $\Co$-algebrable.
\end{remark} 

Remark, that every approximately continuous function $f$ has a point of continuity on any interval (cf. \cite{BBT}). Hence, there is no approximately continuous function that is discontinuous everywhere.

\subsection{Inclusion $\mathcal{C}_{n-1} \supsetneq \mathcal{C}_n$}

In the paper \cite{BBF}, the authors considered continuous functions $f:\R\to\R$ that are differentiable almost everywhere but are not differentiable on any interval. Using Theorem \ref{method BBF}, they obtained the following.

\begin{theorem}
Let $n\in\N\cup\{0\}.$ The set of functions $f\in\mathcal{C}_n$ that are differentiable $n+1$ times almost everywhere, but are not differentiable $n+1$ times on any interval, is strongly $\Co$-algebrable.
\end{theorem}

We will consider a stronger condition of non-differentiability of a function $f:\R\to\R.$ It is well known that the set of all nowhere differentiable functions is comeager in $C[0,1].$ 

Let us recall the notion of nowhere H\"older function.
\begin{definition}
We say that a continuous function $F:[0,1]\to\R$ is nowhere H\"older provided that for any $x\in[0,1]$ and any $\alpha\in(0,1]$ $$\lim_{y\to x}\frac{|F(x)-F(y)|}{|x-y|^\alpha}=\infty.$$ Let us denote the set of all nowhere H\"older functions by $\mathcal{NH}.$
\end{definition}

This kind of continuous function was considered by many authors. J. Kol\'a\v{r} in \cite{Kol}, proved that the set $\mathcal{NH}$ is big in several senses (i.e. it was shown, that $C[0,1]\setminus \mathcal{NH}$ is HP-small, hence it is meager and Haar-null).
It is easy to check that any nowhere H\"older function is nowhere differentiable, hence $\mathcal{NH}\subsetneq \mathcal{C}_0\setminus \mathcal{C}_1.$ Some authors studied spaceability and algebrability of $\mathcal{NH}$. In \cite{Hen}, S. Hencl proved that the set of nowhere approximately differentiable and nowhere H\"older functions is spaceable. Moreover, F. Bayart and L. Quarta in \cite{BQ} showed that $\mathcal{NH}$ is $\omega$-algebrable. Here, following the idea from \cite{BQ} and making use of exponential-like functions, we prove strong $\Co$-algebrability of $\mathcal{NH}.$ At first, let us obtain a general result.

\begin{theorem}\label{aNH}
For any nonconstant analytic function $f:\R\to\R$ and any $F\in\mathcal{NH},$ $f\circ F\in\mathcal{NH}.$
\end{theorem}

\begin{proof}
Let $F:[0,1]\to\R$ be a nowhere H\"older function and let $f:\R\to\R$ be a nonconstant analytic function. Fix a point $x\in[0,1]$ and $\alpha\in(0,1]$. We will prove that $f\circ F$ is not $\alpha$-H\"older at $x.$ Notice, that since $f$ is nonconstant and analytic, there is $k\in\N$ with $f^{(k)}(F(x)) \neq 0,$ where $f^{(k)}(F(x))$ stands for the $k$-th derivative of $f$ at $F(x).$ Assume that $k$ is the least natural number with this property. Let $y\in [0,1],$ by the Taylor expansion of $f\circ F$ in $x,$ we have $$f(F(x))-f(F(y))=f^{(k)}(z)(F(x)-F(y))^k,$$ for some $z$ from the neighborhood of $F(x).$ Therefore, $$\frac{|f(F(x))-f(F(y))|}{|x-y|^\alpha}=|f^{(k)}(z)|\left(\frac{|(F(x)-F(y))|}{|x-y|^{\frac{\alpha}{k}}}\right)^k$$ so by taking the limit with $y\to x,$ we get the assertion.
\end{proof}

One can easily check, that any exponential-like function is analytic and, by Lemma \ref{expLem}, it is nonconstant. Hence, by Theorem \ref{aNH} and Theorem \ref{method BBF}, we obtain:

\begin{corollary}
The set $\mathcal{NH}$ is strongly $\Co$-algebrable.
\end{corollary}

Moreover, we have the following.
\begin{theorem}\label{NH1}
The set of functions from $\mathcal{C}_{1}$ which derivative is not $\alpha$-H\"older (for any $\alpha\in(0,1]$) at all but finite many points, is strongly $\Co$-algebrable.
\end{theorem}

\begin{proof}
Consider a function $G:[0,1]\to\R$ given by $G(x)=\int_{0}^{x} F(t) dt$ where $F\in\mathcal{NH}$ is strictly positive. Clearly, $G\in\mathcal{C}_1$ and its derivative is nowhere H\"older. Let $f:\R\to\R$ be an exponential-like function. Then $f\circ G\in \mathcal{C}_1.$ Let $x\in [0,1]$ be such that $f'(G(x))\neq 0$ and $\alpha\in(0,1].$ For $y\in[0,1]$ we have $$(f\circ G)'(x)-(f\circ G)'(y)=\frac{1}{2}\left[ (f'(G(x))-f'(G(y)))(F(x)+F(y))  + (f'(G(x))-f'(G(y)))(F(x)-F(y))\right].$$ By the Taylor expansion of $f'$ in $G(x),$ we have 
$$(f\circ G)'(x)-(f\circ G)'(y)= \frac{1}{2}\left[ f^{(k)}(z)(G(x)-G(y))^k(F(x)+F(y)) + (f'(G(x))-f'(G(y)))(F(x)-F(y))\right]$$ where $k\in\N$ is the least natural order of nonzero derivative of $f'$ in $G(x)$ and $z$ is some point from the neighborhood of $G(x).$ Therefore, $$\frac{(f\circ G)'(x)-(f\circ G)'(y)}{|x-y|^\alpha}=\frac{1}{2}\left[ f^{(k)}(z)\left(\frac{(G(x)-G(y))}{|x-y|^{\frac{\alpha}{k}}}\right)^k(F(x)+F(y)) + \right. $$ $$\left. + \big(f'(G(x))-f'(G(y))\big)\frac{F(x)-F(y)}{|x-y|^\alpha}\right].$$ By taking the limit with $y\to x$ we have that $$\frac{(f\circ G)'(x)-(f\circ G)'(y)}{|x-y|^\alpha}\to \infty.$$ Hence, $(f\circ G)'$ is not $\alpha$-H\"older at $x.$  
Observe, that since $G$ is one-to-one, there are finitely many $x\in[0,1]$ with $f'(G(x))=0.$ Therefore, $(f\circ G)'$ is not $\alpha$-H\"older (for any $\alpha\in(0,1]$) at all but finite many points. Hence, the use of Theorem \ref{method BBF} gives the assertion.
\end{proof}

\subsection{Inclusion $\mathcal{C}_\infty \supsetneq \mathcal{A}$}

Let us call a function $g:\R\to\R$ \emph{nowhere analytic}, when it is not analytic at any point. The set of smooth and nowhere analytic functions was considered in the context of algebrability in \cite{GonzAN}, where spaceability of this set was established. 

\begin{theorem}\label{na}
The set of smooth and nowhere analytic functions is strongly $\Co$-algebrable.
\end{theorem}

\begin{proof}
Let $F:\R\to\R$ be smooth and nowhere analytic (cf. \cite{GonzAN}) and $f:\R\to\R$ be an exponential-like function. Suppose, on the contrary, that $g=f\circ F$ is analytic at a point $x_0\in\R.$ Then there is an open set $V\ni x_0$ such that $g$ is analytic on $V$ (cf. \cite{Anal}). By Lemma \ref{expLem} there is an open set $U\subseteq F(V)$ on which $f$ is invertible. Hence, $F=f^{-1}\circ g$ is analytic (as a composition of analytic functions, see \cite{Anal}) on the set $F^{-1}(U)\cap V,$ that yields us to a contradiction. By an application of Theorem \ref{method BBF}, the proof is finished.
\end{proof}

\subsection{Nowhere monotone and Bruckner-Garg functions}

We will say, in the sequel, that a subset of $[0,1]$ is a Cantor set, provided it is homeomorphic to the ternary Cantor set.
Let us move our considerations to the set of continuous functions $f:[0,1]\to\R$ and let us introduce the following notion.

\begin{definition}
We say that a continuous function $f:[0,1]\to\R$ is:
\begin{itemize}
\item differentiable nowhere monotone (shortly $f\in \mathcal{DNM}$), provided $f$ is differentiable and is not monotone on any interval $[a,b]\subseteq[0,1];$
\item Bruckner-Garg of rank $k\in\N$ (shortly $f\in \mathcal{BG}_k$), provided there exists a countable set $A\subseteq (\min f, \max f)$ with the property that for all $x\in A$ the preimage $f^{-1}(\{x\})$ is an union of a Cantor set with at most $k$ many isolated points, and for all $x\in(\min f,\max f)\setminus A$ the preimage $f^{-1}(\{x\})$ is a Cantor set;
\item Bruckner-Garg (shortly $f\in\mathcal{BG}_\omega$), provided it is Bruckner-Garg of rank $k$ for some $k\in\N.$
\end{itemize}
\end{definition}

Algebrability of differentiable nowhere monotone functions was considered in \cite{2}, where $\omega$-lineability of $\mathcal{DNM}$ was established. This result was strengthened in \cite{8} to $\Co$-lineability. Here, using exponential-like functions, we may obtain the following.

\begin{theorem}
The set $\mathcal{DNM}$ is strongly $\Co$-algebrable.
\end{theorem}

\begin{proof}
Let $F:[0,1]\to\R$ be a differentiable nowhere monotone function and let $f:\R\to\R$ be an exponential-like function. Obviously, $f\circ F$ is differentiable. To show that $f\circ F$ is nowhere monotone, suppose there exists an interval $[a,b]\subseteq [0,1]$ such that $f\circ F$ is monotone on it. Then, by Lemma \ref{expLem}, there is an interval $[c,d]\subseteq [a,b]$ such that $f$ is monotone on $[F(c),F(d)].$ But then $F$ is monotone on $[c, d]$ and we reach a contradiction. Hence, an application of Theorem \ref{method BBF} gives us the thesis.
\end{proof}

Bruckner-Garg functions of rank $1$ was investigated in \cite{BrGa}, where it was shown that $\mathcal{BG}_1$ is residual. It is an open question if $\mathcal{BG}_1$ is Haar-null in the sense of Christensen (cf. \cite{Chris}). Here, we present a result in algebrability of $\mathcal{BG}_\omega.$

\begin{theorem}
The set $\mathcal{BG}_\omega$ is strongly $\Co$-algebrable.
\end{theorem}

\begin{proof}
Let $F\in\mathcal{BG}_1\subseteq\mathcal{BG}_\omega$ (cf. \cite{BrGa}) and let $A\subseteq (\min F, \max F)$ be a countable set that exists by the definition of Bruckner-Garg function of rank $1.$ Let $f:\R\to\R$ be an exponential-like function. Obviously $f\circ F$ is continuous. To show that $f\circ F\in \mathcal{BG}_\omega$, consider a countable set $B=f(A).$ Notice that $(f\circ F)^{-1}(\{x\})=F^{-1}(f^{-1}(\{x\})).$ Let $x\in B.$ By Lemma \ref{expLem}, $f^{-1}(\{x\})$ is a finite set with at least one point from $A,$ so $F^{-1}(f^{-1}(\{x\}))$ is a finite union of Cantor sets and finitely many isolated points. Hence it is homeomorphic to the ternary Cantor set with finitely many isolated points. Fix now $x\in (\min (f\circ F), \max (f\circ F))\setminus B,$ then $f^{-1}(\{x\})$ is finite and disjoint with $A$. Therefore, $(f\circ F)^{-1}(\{x\})$ is a finite union of Cantor sets, hence it is homeomorphic to the ternary Cantor set. By the application of Theorem \ref{method BBF}, we are done.
\end{proof}

\section{The Baire functions}

This section will be devoted to the classes of real functions connected with the Baire hierarchy.

Let us recall the following notion.
\begin{definition}
We say that a function $f:\R\to\R$ is
\begin{itemize}
\item of Baire class one (shortly $f\in\mathcal{B}_1$), provided it is a pointwise limit of a sequence of continuous functions;
\item of Baire class $\beta$ (shortly $f\in \mathcal{B}_\beta$), provided it is a pointwise limit of a sequence of functions from $\bigcup_{\alpha<\beta} \B_\alpha$ (where $\beta$ is an ordinal and $1<\beta<\omega_1$);
\item of the class $\mathcal{DB}_1$, provided $f$ is of Baire class one and has the Darboux property.
\end{itemize}
\end{definition}

For any ordinal number $\beta<\omega_1$ we have $\mathcal{AP} \subsetneq \mathcal{DB}_1 \subsetneq \mathcal{B}_1 \subsetneq ... \subseteq \mathcal{B}_\beta$.  
It is a consequence of Lebesgue Theorem, that all of the inclusions connected with Baire hierarchy are strict (cf. \cite{Sri}).
Let us recall the following notion. For any topological space $X$ and any family $\mathcal{F}\subseteq \mathcal{P}(X)$, let $\mathcal{F}_\sigma = \{\bigcup F_n : F_1,F_2,...\in\mathcal{F}\}$ and $\mathcal{F}_\delta = \{\bigcap F_n : F_1,F_2,...\in\mathcal{F}\}.$

Let $\Sigma^0_1(X)$ and $\Pi^0_1(X)$ stand for the family of all open and closed sets in $X$, respectively. By $\Sigma^0_0(X)=\Pi^0_0(X)$ we understand the family of all clopen sets in $X$.
For any ordinal number $1<\beta<\omega_1$ let us define $\Sigma^0_\beta(X)=\Big(\bigcup_{\alpha<\beta}\Pi^0_\alpha(X)\Big)_\sigma$ and $\Pi^0_\beta(X)=\Big(\bigcup_{\alpha<\beta}\Sigma^0_\alpha(X)\Big)_\delta$. 
In the sequel, we consider $X=\R$ with natural topology and we write shortly $\Sigma^0_\beta=\Sigma^0_\beta(\R), \Pi^0_\beta=\Pi^0_\beta(\R),$ for any ordinal number $\beta<\omega_1.$
It is well known that $\bigcup_{\beta<\omega_1} \Sigma^0_\beta=\bigcup_{\beta<\omega_1} \Pi^0_\beta$ is the family of all Borel sets in $\R.$

There is a classical criterion of being a function of a certain Baire class (cf. \cite{Sri}).

\begin{theorem}
Let $F:\R\to\R$ and let $\beta<\omega_1.$ The following conditions are equivalent
\begin{itemize}
\item[$(i)$] $F\in \mathcal{B}_\beta;$
\item[$(ii)$] $F^{-1}(V)\in \Sigma^0_{\beta+1}$ for every open set $V\subseteq \R.$
\end{itemize} 
\end{theorem}

We will prove strong $\Co$-algebrability of all of the differences which appear in the list of inclusions, mentioned at the beginning of the section.

\begin{theorem}
The set $\mathcal{DB}_1\setminus \mathcal{AP}$ is strongly $\Co$-algebrable.
\end{theorem}

\begin{proof}
Let $a_n=\frac{1}{2^{2n-2}}$, $b_n=\frac{1}{2^{2n-1}}$, $c_n=\frac{a_n+b_n}{2}$, $d_n=\frac{a_{n+1}+b_n}{2},$ for $n\in \N,$ and let $(y_n)_{n\in\N}$ be a numeration of $\{\frac{n-1}{n} : n\in\N\}$, such that each term occurs infinitely many times. Consider a function $F:[0,1]\to\R$, defined by 
$$F(x)=\begin{cases} 0 \text{, when } x\in\bigcup[d_n,b_n] \cup \{0\} \\ y_n \text{, when } x\in [c_n,a_n] \\ \frac{y_n(x-b_n)}{c_n-b_n} \text{, when } x\in (b_n,c_n) \\ \frac{y_{n+1}(x-d_n)}{a_{n+1}-d_n} \text{, when } x\in(a_{n+1},d_n). \end{cases}$$ Clearly $F$ has the Darboux property. Moreover, it is easy to check that preimage by $F$ of any open set is $\Sigma^0_2$, hence $F\in\mathcal{DB}_1.$ Notice that for any $n\in\N$ we have $\bar{d}^+\big(F^{-1}\big(\{\frac{n-1}{n}\}\big),0\big)\geq \frac{1}{4}$ where $\bar{d}^+(A,x)$ stands for the upper Lebesgue density of a set $A$ at the point $x.$ Hence, $F\in\mathcal{DB}_1\setminus \mathcal{AP}$ . Let $f:[0,1]\to\R$ be an exponential-like function. By Lemma \ref{expLem} there exists $n\in\N$ with $f(\frac{n-1}{n})\neq f(0).$ Notice that then $\bar{d}^+\big((f\circ F)^{-1}\big(f(\frac{n-1}{n})\big),0\big)\geq \frac{1}{4}$. Hence, $f\circ F$ is not approximately continuous. On the other hand, since $f$ is continuous, $f\circ F\in\mathcal{DB}_1$ and an applications of Theorem \ref{method BBF} finishes the proof.
\end{proof}

\begin{theorem}
The set $\mathcal{B}_1\setminus \mathcal{DB}_1$ is strongly $\Co$-algebrable.
\end{theorem}

\begin{proof}
Consider a function $F:[0,1]\to[0,1]$ defined by $$F(x)=\begin{cases} 0 \text{, when } x=0 \\ \frac{1}{n} \text{, when } x\in \left(\frac{1}{2^n},\frac{1}{2^{n-1}}\right].\end{cases}$$ Clearly, $F$ does not have the Darboux property. Moreover, preimage by $F$ of any open set is either empty or is a countable union of intervals, hence is $\Sigma^0_2$ set. So $F\in \mathcal{B}_1 \setminus \mathcal{DB}_1.$ Let $f:[0,1]\to\R$ be an exponential-like function. Then clearly $f \circ F$ is of the Baire class one. Moreover, by Lemma \ref{expLem}, it does not have the Darboux property (indeed, if $f$ is of rank $m\in\N$ then $(f\circ F)\left(\left(\frac{1}{2^{m+1}},1\right]\right)$ is finite and has at least two elements - i.e. is not connected). Hence, by Theorem \ref{method BBF} we are done.
\end{proof}

Let us move now to the Baire classes. 

\begin{theorem}
For any ordinal number $\beta<\omega_1,$ the set of all functions $F\in\mathcal{B}_\beta$ such that for any nonempty open set $U\subseteq \R$ the restriction $F|_U\notin \bigcup_{\alpha<\beta}\mathcal{B}_\alpha$ is strongly $\Co$-algebrable.
\end{theorem}

\begin{proof}
Let $\beta<\omega_1$ and let $(V_n)_{n\in\N}$ be a basis in $\R.$ Inductively, we define a sequence of pairwise disjoint Cantor sets $(C_n)_{n\in\N}$ such that $C_n\subseteq V_n$ for any $n\in\N.$ Moreover, for any $n\in\N$ one can find a collection of sets $A_{n, k}\subseteq C_n,$ for $k\in\{n,n+1,...\},$ such that $A_{n,k}\in \Pi^0_\beta \setminus \Sigma^0_\beta$ (cf. \cite{Sri}). Let $A_k=A_{1,k}\cup A_{2,k} \cup ... \cup A_{k,k},$ for $k\in\N,$ then clearly $A_k\in\Pi^0_\beta \setminus \Sigma^0_\beta.$ Define a function $F:\R\to\R$ by  $F(x)=\sum_{k=1}^{\infty}\frac{1}{k}\chi_{A_k}(x)$. To see that $F\in\mathcal{B}_\beta$, consider an open set $V\subseteq \R.$ Then $F^{-1}(V)$  is either empty or is a countable union of sets from $\{A_1,A_2,...\}\subseteq \Pi^0_\beta \setminus \Sigma^0_\beta,$ hence $F^{-1}(V)\in \Sigma^0_{\beta+1}.$ Moreover, for any nonempty open set $U\subseteq \R$ there is $k\in\N$ with $V_k\subseteq U$ and there is $h>0$ with $F|_U^{-1}(\frac{1}{k}-h,\frac{1}{k}+h)=F|_U^{-1}(\{\frac{1}{k}\})=A_k \cap U \notin \Sigma^0_\beta,$ hence $F|_U \notin \bigcup_{\alpha<\beta}\mathcal{B}_\alpha.$
Let $f:\R\to\R$ be an exponential-like function. Since $f$ is continuous, $f\circ F\in \mathcal{B}_\beta.$ Let $U\subseteq \R$ be a fixed nonempty open set. By Lemma \ref{expLem} there is $\varepsilon >0$ such that $f|_{(0,\varepsilon)}$ is a homeomorphism. Moreover, there is $k\in\N$ and $h>0$ with $V_k\subseteq U$ and $(\frac{1}{k}-h,\frac{1}{k}+h)\subseteq (0,\varepsilon)\setminus \{\frac{1}{k+1},\frac{1}{k-1}\}.$ Consider an open set $f((\frac{1}{k}-h,\frac{1}{k}+h)),$ we have that $(f\circ F)|_U^{-1}f((\frac{1}{k}-h,\frac{1}{k}+h))=F^{-1}((\frac{1}{k}-h,\frac{1}{k}+h))\cap U = A_k \cap U \notin \Sigma^0_\beta$. Hence, $(f\circ F)|_U\notin \bigcup_{\alpha<\beta}\mathcal{B}_\alpha$ and due to Theorem \ref{method BBF}, the desired set is strongly $\Co$-algebrable.
\end{proof}

Notice that exponential-like functions can be also used in proving algebrability of some subsets of $\R^\N$. Consider the following notion: for a sequence $x\in\ell_\infty,$  (space of all bounded real sequences), let us denote by $\on{LIM}(x)=\{y\in\R : x(n_k)\to y, \text{ for some increasing } (n_k)_{k\in\N}\}$ (i.e. the set of all limit points of $x$). In general, $\on{LIM}(x)$ is a nonempty compact subset of $\R.$ In \cite{6} the authors proved strong $\Co$-algebrability of the set of bounded real sequences $x,$ for which the set $\on{LIM}(x)$ is an union of Cantor set and finitely many isolated points. Here, due to Theorem \ref{method BBF}, we obtain a stronger result.

\begin{theorem}
The set of bounded real sequences $x,$ for which the set $\on{LIM}(x)$ is a Cantor set, is strongly $\Co$-algebrable.
\end{theorem}

\begin{proof}
Let $x=(x(n))_{n\in\N} \in \ell_\infty$ be such that $\on{LIM}(x)$ is a Cantor set of Lebesgue measure zero and let $f:\R\to\R$ be an exponential-like function. We will show that for $f\circ x :\N\to\R,$ the set $LIM(f\circ x)$ is a Cantor set. By continuity of $f,$ it is easy to check that $\on{LIM}(f(x))=f(\on{LIM}(x)).$ Moreover, $f$ is finite-to-one (by Lemma \ref{expLem}) and satisfies the (N) Luzin condition, i.e. it maps Lebesgue null sets to Lebesgue null sets. Hence, the set $f(\on{LIM}(x))$ is a Cantor set (cf. \cite{6}). Now, it is enough to apply Theorem \ref{method BBF}.
\end{proof}

\section{Final remarks and open questions}

This paper shows that the use of compositions with exponential-like functions has a lot of applications in algebrability (both in spaces of functions and of sequences). However, this idea is not \textbf{general} in the sense, that using it we cannot prove strong $\Co$-algebrability always when it is possible. One can consider a set $\mathcal{K}$ of all smooth functions $F:\R\to\R$ with compact supports. Let $F_0:\R\to\R$ be defined by $$F_0(x)=\begin{cases} \exp{\frac{-1}{1-x^2}} \text{, if } |x|<1 ;\\ 0 \text{, elsewhere} \end{cases},$$ then $F_0\in \mathcal{K}.$ Let $H$ be a set of cardinality $\Co$, linearly independent over $\Q.$ It is rather easy to see that $\{F_0(x) \exp{(rx)}:r\in H\}$ generates a free subalgebra of $\mathcal{K}\cup\{0\}.$ Hence $\mathcal{K}$ is strongly $\Co$-algebrable. On the other hand, for an exponential-like function $f(x)=\exp{(x)}$ and any $F\in\mathcal{K},$ we have that $0\neq f\circ F \notin \mathcal{K}$ (since it has a full support, in particular non compact). So Theorem \ref{method BBF} does not work in that case.

On the other hand, despite the fact that almost all of the presented proofs are relatively simple, they essentially improve some known results in algebrability and rise them to the highest possible level.

To show another advantage of the use of compositions with exponential-like functions, let us consider the following notion: for a Banach algebra $E$ we say that a set $\mathcal{A}\subseteq E$ is \emph{densely strongly $\kappa$-algebrable}, provided there exists a free algebra of $\kappa$-generators contained in $\mathcal{A}\cup \{0\}$ which is dense in $E.$ In \cite{BBF}, the authors showed that the set of strongly singular functions is densely strongly $\Co$-algebrable in $C[0,1].$ One can ask, when an algebra of continuous functions, constructed as in the proof of Theorem \ref{method BBF} is dense in $C[0,1].$ The answer to this is that the constructed algebra is dense, if and only if the starting function $F$ is one-to-one. Indeed, if $F$ is one-to-one then by similar argument as in Theorem 10 in \cite{BBF} the algebra is dense. On the other hand, when $F$ is not one-to-one (i.e. it does not separate some points $x$ and $y$) then any function in the algebra (even in its closure) cannot separate $x$ and $y.$ By this observation and the fact, that there exists a strictly monotone nowhere analytic smooth function, we may obtain a dense strong $\Co$-algebrability of the set in Theorem \ref{na}.

Let us remark also, that $F\in \mathcal{F},$ as an assumption in Theorem \ref{method BBF}, can be omitted. For an existence of a free algebra of $\Co$ generators in $\mathcal{F}$, it is enough that $f\circ F\in \mathcal{F}\setminus \{0\},$ for any exponential-like function $f.$ Moreover, the following example shows that a set is strongly $\Co$-algebrable, what can be proved by starting only from a function $F\notin \mathcal{F}.$ Consider a set $\mathcal{EXP}$ of all exponential-like functions $f:\R\to\R.$ Let $F=\on{id}:\R\to\R,$ clearly $F\notin\mathcal{EXP}.$ Obviously, $f\circ F\in \mathcal{EXP},$ for any exponential-like function $f:\R\to\R.$ Hence, by the proof of Theorem \ref{method BBF}, $\mathcal{EXP}$ is strongly $\Co$-algebrable. On the other hand, a starting function $F$ cannot be taken from $\mathcal{EXP},$ since then $f\circ F \notin \mathcal{EXP}.$

At the end, let us state some open questions connected with algebrability and classes of function considered in this paper.

\begin{problem}
Is the set of all functions from $\mathcal{C}_{n}$ which the $n$-th derivative is nowhere H\"older strongly $\Co$-algebrable?
\end{problem}

\begin{problem}
Is the set $\mathcal{BG}_1$ algebrable?
\end{problem}

\end{document}